\documentclass[12pt, a4paper, reqno]{amsart}
\usepackage{latexsym,amsmath, amssymb, eucal, amscd, amstext, enumerate, mathrsfs, yfonts, amsfonts,comment, verbatim, color}
\usepackage[plainpages=false,colorlinks,pdfpagemode=None,bookmarksopen,linkcolor=blue,citecolor=blue,urlcolor=blue]{hyperref}
\usepackage{cite}

\def\tto{\;{\lower 1pt \hbox{$\rightarrow$}}\kern -10pt
\hbox{\raise 2pt \hbox{$\rightarrow$}}\;}

\newtheorem{Theorem}{Theorem}[section]
\newtheorem{Proposition}[Theorem]{Proposition}
\newtheorem{Lemma}[Theorem]{Lemma}
\newtheorem{Corollary}[Theorem]{Corollary}

\theoremstyle{definition}

\theoremstyle{remark}
\newtheorem{Remark}[Theorem]{Remark}
\newtheorem{Example}[Theorem]{Example}
\newtheorem{question}[Theorem]{Question}

\numberwithin{equation}{section}

\begin{document}

\title[Super asymptotically nonexpan\-sive actions on Fr\'{e}chet spaces]{Super asymptotically nonexpan\-sive actions of semitopological semigroups on Fr\'{e}chet and locally convex spaces}

\author[Muoi and Wong]{Bui Ngoc Muoi \and {Ngai-Ching Wong}}

\address{Department of Applied Mathematics, National Sun Yat-sen University, Kaohsiung, 80424, Taiwan.}
\email{muoisp2@gmail.com, wong@math.nsysu.edu.tw}

%\thanks{This research is supported by the Taiwan MOST grant 108-2115-M-110-004-MY2.}

\thanks{Corresponding author: Ngai-Ching Wong, E-mail: wong@math.nsysu.edu.tw}
\date{\today}% ver. Apr29c 2020, submitted to JNCA
%\thanks{This research is supported by the Taiwan MOST grant 108-2115-M-110-004-MY2.}
\dedicatory{Dedicated to Professor Wataru Takahashi on the occasion of  his 75th birthday}

\begin{abstract}
Let  $\operatorname{LUC}(S)$ be the space of left uniformly continuous functions on
a  semitopological semigroup $S$.
Suppose that $S$ is right reversible and  $\operatorname{LUC}(S)$ has a left invariant mean.
Let $(X,d)$ be a Fr\'{e}chet space. Let $\tau$ be a locally convex topology of $X$ weaker than the $d$-topology such
that the metric $d$ is $\tau$-lower semicontinuous.
Let $K$ be a $d$--separable and $\tau$--compact convex subset of $X$.
We show that every jointly  $\tau$-continuous and super asymptotically $d$--nonexpan\-sive action $S\times K\mapsto K$ of $S$ has
 a common fixed point.
Similar results in the locally convex space setting are provided.
\end{abstract}



\date{}% a revision dated May 8, 2020.

\maketitle

\section{Introduction}

Let $S$ be a \emph{semitopological semigroup}, i.e., $S$ is a semigroup with a (Hausdorff) topology such that for each fixed $t\in S$, the mappings $s\mapsto ts$ and $s\mapsto st$ are continuous.
An \emph{action} of $S$ on a  topological space $K$ is a mapping of $S\times K$ into $K$, denoted by
$(s,x) \mapsto s.x$, such that $(st).x=s.(t.x)$ for all $s, t\in S$ and $x\in K$.  We call the action  \emph{separately} (resp.\ \emph{jointly})
\emph{continuous} if the mapping $(s, x)\mapsto s.x$ is separately (resp.\ jointly) continuous. A point $x_0\in K$ is called
a \emph{common fixed point} for $S$ if $s.x_0=x_0$ for all $s\in S$.

In 1976, Karlovitz \cite[page 322]{Karlov76} proved that if $K$ is a weak* compact convex subset of $l_1$ (which is the Banach dual of $c_0$) then
every nonexpan\-sive mapping from $K$ to itself has a fixed point. Indeed, a commutative family of weak* continuous nonexpan\-sive mappings on a weak*
compact convex subset of a Banach dual space always attains a common fixed point (see \cite[Theorem 3.5]{BW2014}).
However, an affine nonexpan\-sive mapping on a weak* compact convex subset of a general Banach dual space does not always have a fixed point  (see
e.g. \cite[Example 3.2]{Sims01}).  This indicates that we need, in general, both the topological and
the uniform properties  to ensure the existence of a  fixed point.

Observe that every (Hausdorff) locally convex space carries a uniform structure defined by continuous seminorms.
More precisely, we write $(X,Q)$
for a locally convex space $X$ with  a family $Q$
of selected seminorms defining its topology.
Extending from the Banach space version (see \cite{MuoiWong2020}) we say that
an action of a semitopological semigroup $S$ on a subset $K$ of  $(X, Q)$ is
\begin{itemize}\itemsep=6pt
	\item \textit{asymptotically uniformly $Q$--nonexpan\-sive} (see  \cite{HolLau71}) if for    $x, y\in K$ there exists a left ideal $I_{xy}$ of $S$ such that $q(s.x-s.y)\leq q(x-y), \forall s\in I_{xy},\,
	 \forall q\in Q$;

	\item \textit{asymptotically separately $Q$--nonexpan\-sive}  if for    $x, y\in K$   and
	   $q\in Q$  there exists a left ideal $I^q_{xy}$ of $S$ such that $q(s.x-s.y)\leq q(x-y), \forall s\in I^q_{xy}$;

	\item \textit{super asymptotically uniformly $Q$--nonexpan\-sive} if for   $x\in K$ and    $t\in S$ there exists a left ideal $I^t_x$ of $S$ such that $q(st.x-st.y)\leq q(x-y), \forall s\in I^t_{x},\ \forall y\in K,\,
	 \forall q\in Q$;

	\item \textit{super asymptotically separately $Q$--nonexpan\-sive} if for     $x\in K$,   $t\in S$  and
	   $q\in Q$ there exists a left ideal $I^{q,t}_{x}$  of $S$ such that $q(st.x-st.y)\leq q(x-y), \forall s\in I^{q,t}_{x},\,
\forall y\in K$.
\end{itemize}
Note that these uniformity notions depend on the seminorms in $Q$ but not the topology $Q$ defining.
\begin{comment}
For example, we can enlarge $Q$ by adding all seminorms $x\mapsto |f(x)|$
  arising from continuous linear functionals $f$.  Then the Hahn-Banach theorem ensures that every asymptotically $Q$-nonexpansive action
 is trivial in the sense that all $s.x = x_0$ for a fixed point $x_0$.
\end{comment}

 When  $K$ is a  metric space with metric $d$, we call the action
\begin{itemize}\itemsep=6pt%[{1.}]\addtocounter{enumi}{4}	
\item  \textit{asymptotically $d$--nonexpan\-sive} (see \cite{HolNar70}) if for each $x,y\in K$, there exists a left ideal $I_{xy}$ of $S$ such that $d(s.x,s.y)\leq d(x,y)$ for all  $s\in I_{xy}$;
	\item  \textit{super asymptotically $d$--nonexpan\-sive} if for each $x\in K$ and $t\in S$, there exists a left ideal $I^t_x$ of $S$ such that $d(st.x,st.y)\leq d(x,y)$ for all $y\in K$ and $s\in I^t_x$.
	\end{itemize}

Let $X$ be a metrizable locally convex space.
Let $Q=\{q_n: n=1,2,\ldots\}$ be a countable family of seminorms defining the topology of
 $X$, and define a metric
\begin{equation}\label{sumMetric}
d(x,y)=\sum_{n=1}^{\infty}\frac{1}{2^n}\frac{q_n(x-y)}{1+q_n(x-y)}.
\end{equation}
It is easy to see that every asymptotically (resp.\ super asymptotically) uniformly
$Q$--nonexpan\-sive action  on a subset $K$ of $X$ is asymptotically (resp.\ super asymptotically)
$d$--nonexpan\-sive.  Conversely, let $d$ be a translation invariant metric defining the topology
of  $X$ such that all open metric balls are absolutely convex.  As shown in Remark \ref{rem:d->unifQ},
there is a family $Q$ of continuous seminorms defining the metric topology of
$X$ such that   asymptotically (resp.\ super asymptotically)
$d$--nonexpan\-sive actions  on a subset $K$ of $X$ are exactly those being asymptotically (resp.\ super asymptotically)
uniformly $Q$--nonexpan\-sive.  Consequently, the fixed point theory for various uniformly $Q$-nonexpan\-sive actions
is equivalent to that for the $d$-nonexpan\-sive ones (see Corollary \ref{cor:unif-Q}).

In Section \ref{Sect2}, we consider the case the topological structure and the uniform structure of $K$  arising from
different context.  More precisely, let $(X,\tau)$ be a locally convex space.  Assume in addition that
 there is a  translation invariant  metric $d$ on $X$ making
$(X,d)$   a Fr\'{e}chet space.
Moreover, the $\tau$-topology is weaker than the $d$-topology
and $d$ is $\tau$-lower semicontinuous.  Following \cite{LauTaka09},
we call $\tau$ a \emph{$d$--admissible} topology in this case.
For example, the weak (resp.\ weak*) topology of a Banach space (resp. Banach dual space) is admissible with respect to the
norm metric.

Let $K$ be a $d$--separable and $\tau$--compact convex subset of  $X$.
We show in Theorem \ref{mainThm}, among other results, that if $S$ is a right reversible semitopological semigroup such that
$\operatorname{LUC}(S)$  has  a  left invariant mean (for definitions see Section \ref{Sect2}),
then every jointly $\tau$--continuous and super asymptotically $d$--nonexpan\-sive action $S\times K\mapsto K$ of $S$ on  $K$   has a common
fixed point. This extends earlier results of the authors \cite{MuoiWong2020}, and also supplements other results in literature, e.g., \cite{AMN2016,KidoTaka84,LauTaka95,Taka81}, from the Banach space setting to
the Fr\'{e}chet space setting.

In Section \ref{Sect3},  we  provide some fixed point theorems for super asymptotically separately $Q$--nonexpan\-sive actions on a
locally convex space $(X,Q)$.  In this setting,  both the uniform structure and the topological structure arise from the
same family $Q$
of continuous seminorms of $X$.
To end this paper, we provide examples to demonstrate that the notions of
asymptotic nonexpan\-siveness and super asymptotic nonexpan\-siveness can be different,
while they coincide for separately continuous actions of right reversible compact semigroups.

\section{Asymptotically nonexpansive actions on Fr\'{e}chet spaces}\label{Sect2}

Let $S$ be a semitopological semigroup throughout.
We say that  $S$ is \textit{right (resp.\ left) reversible} if any two closed left (resp.\ right) ideals of $S$ intersect.
Denote by $l^{\infty}(S)$ (resp.\ $\operatorname{CB}(S)$) the Banach space of   bounded
(resp.\ bounded and continuous) real-valued functions on $S$ with the supremum norm.
For each $s\in S$ and $f\in l^{\infty}(S)$, we denote by $l_sf$ the left translation of $f$ by $s$
where $l_sf(t)=f(st)$ for all $t\in S$. A subspace $X\subseteq l^{\infty}(S)$ that contains all constant functions is called
\textit{left translation invariant} if $l_s(X)\subseteq X$ for all $s\in S$. A linear functional $m\in X^*$ is called a \emph{mean}
if $\|m\|=m(1)=1$.
A mean $m$ is called a \textit{left invariant mean}, or $\operatorname{LIM}$ in short,
if $m(l_sf)=m(f)$ for all $s\in S$ and $f\in X$.   We also have symmetric concepts about right invariant means.  An
\emph{invariant mean} of $X$
is a mean which is both left and right invariant.

Let $\operatorname{LUC}(S)$ be the space of   \textit{left uniformly continuous} functions on $S$; namely those $f$ for which
the map $s\mapsto l_sf$ from $S$ into $\operatorname{CB}(S)$ is norm continuous. The space of \textit{right uniformly continuous} functions $\operatorname{RUC(S)}$ is defined symmetrically.
When $S$ is topological group, $\operatorname{LUC(S)}$ and $\operatorname{RUC(S)}$ coincide (see e.g. \cite[Theorem 15.4]{ER97}).
Let $\operatorname{AP}(S)$ (resp.\ $\operatorname{WAP}(S)$) be the subspace of \textit{almost periodic}
(resp.\ \textit{weakly almost periodic}) functions  in $\operatorname{CB}(S)$;
namely those $f$ for which $\left\lbrace l_sf: s\in S\right\rbrace$ is relatively compact in the norm (resp.\ weak) topology of $\operatorname{CB}(S)$.
All $\operatorname{LUC}(S), \operatorname{AP}(S), \operatorname{WAP}(S)$ are left translation invariant subspaces
of $\operatorname{CB}(S)$. In general, we have
$$
\operatorname{AP}(S)\subseteq \operatorname{LUC}(S)\subseteq \operatorname{CB}(S)\quad\text{and}\quad
\operatorname{AP}(S)\subseteq \operatorname{WAP}(S)\subseteq \operatorname{CB}(S).
$$
When $S$ is compact, we have $\operatorname{AP}(S)=\operatorname{LUC}(S)\subseteq\operatorname{WAP}(S)=\operatorname{CB}(S)$ (see, e.g., \cite[page 2952]{LauZhang12}).

We recall that a
locally convex space $X$ is metrizable if and only if there is a countable family of seminorms
defining its topology.  In this case, we can choose a  ``\textit{good}'' metric $d$ to define the topology of $X$ such that
\begin{itemize}
	\item [(i)] $d(x+y,x+z)=d(y,z)$ for all $x, y, z$ in $X$,
%	\item [(ii)] $d(kx,ky)\leq d(x,y)$ for all $|k|\leq 1$,
	\item [(ii)] $B_r(0)=\left\lbrace x\in X: d(x,0)<r\right\rbrace$ is open and absolutely convex for all $r>0$.
\end{itemize}
The metric topology coincides with the topology defined by the countable family $Q$ of    seminorms
\begin{equation}\label{MinkowskiSeminorm}
q_r(x)=\inf\left\lbrace\lambda>0: x\in\lambda \bar{B}_r(0)\right\rbrace,\quad   r\in \mathbb{Q}\cap (0,+\infty),
\end{equation}
where $\bar{B}_r(0)=\{x\in X: d(x,0)\leq r\}$.
 In this case, we can recover the metric $d$ by
 \begin{equation}\label{eq:d=q_r}
 d(x,y)=\inf\left\lbrace r\in \mathbb{Q}\cap (0,+\infty): q_r(x-y)\leq 1\right\rbrace,
 \end{equation}
 when \eqref{sumMetric} gives rise to an equivalent metric instead (see Rudin \cite[Theorem 1.24]{Rudin91}).
Clearly, the norm metric $d(x,y)=\|x-y\|$ of a Banach space is a ``good'' metric.

\begin{Remark}\label{rem:d->unifQ}
Let $d$ be a ``good'' metric of a Fr\'{e}chet space $X$, and   $Q$ be the countable family of seminorms defined by \eqref{MinkowskiSeminorm}.
Then  (resp.\ super) asymptotically $d$--nonexpan\-sive actions of a semigroup $S$ on a subset $K$ of $X$  are exactly those being
(resp.\ super) asymptotically uniformly $Q$--nonexpan\-sive.

Indeed, suppose for any $x\in K$ and $t\in S$ there exists a left ideal $I^t_x$ of $S$ such that
$$
d(st.x,st.y)\leq d(x,y),\quad\forall s\in I^t_x,\, \forall y\in K.
$$
This gives
$$
x-y\in\bar{B}_r(0)\quad\implies\quad st.x-st.y\in\bar{B}_r(0), \quad \forall r\in\mathbb{Q}\cap (0,+\infty).
$$
In other words,
$$
q_r(x-y)\leq 1\quad\implies\quad q_r(st.x-st.y)\leq 1, \quad\forall r\in\mathbb{Q}\cap (0,+\infty).
$$
Therefore,
\begin{align}\label{eq:Q-nonexp}
q_r(st.x-st.y)\leq q_r(x-y),\quad \forall s\in I^t_x,\, \forall y\in K,\, \forall q_r\in Q.
\end{align}
Conversely, it follows from \eqref{eq:d=q_r} and \eqref{eq:Q-nonexp} that any   super  asymptotically uniformly $Q$--nonexpan\-sive is
 super  asymptotically $d$--nonexpan\-sive.

The case for asymptotically $d$-nonexpansive actions is similar.
\end{Remark}

\begin{Theorem}\label{mainThm}
	Let $S$ be a right reversible semitopological semigroup such that
 $\operatorname{LUC}(S)$ has a left invariant mean. Let $(X,d)$ be a  Fr\'{e}chet space with a ``good'' metric $d$.
Let $\tau$ be a $d$--admissible locally convex topology of $X$.
Then $S$ has the following fixed point property.
	\begin{quote}
	 $\mathbf{(F^{\sup}_{jc,dsep})}$ Every  super asymptotically $d$--nonexpan\-sive and jointly $\tau$ continuous  action of  $S$
 on a $d$--separable and $\tau$--compact convex    subset $K$ of $X$  has a common fixed point.
	\end{quote}
\end{Theorem}

We note that in the setting of Theorem \ref{mainThm}, the left ideal in the definition of
the super asymptotic $d$--nonexpan\-siveness can be assumed closed. Indeed, suppose that for any $x\in Y$, there exists
for each $t\in S$ a left ideal $I^t_{x}$ of $S$ such that $d(stx,sty)\leq d(x,y)$ for all $s\in I^t_x$ and all $y\in Y$. For each
$s_0\in\overline{I^t_x}$, there is a net $\{s_\lambda\}$ in $I^t_x$ converging to $s_0$. Then $\{s_\lambda t x\}$ converges to $s_0tx$ in
the $\tau$ topology. Since $\tau$ is $d$--admissible,
$$d(s_0tx,s_0ty)\leq\liminf_\lambda d(s_\lambda tx,s_\lambda ty)\leq d(x,y).$$
Consequently, the said left ideal can be chosen to be $\overline{I^t_x}$.

The proof of  Theorem \ref{mainThm} needs several lemmas. The first one borrows from the proof of \cite[Theorem 3.1]{HolLau71}.
\begin{Lemma}\label{lemma1}
	Let $S$ be a right reversible semitopological semigroup.  Assume $S\times K\to K$ is a
separately  continuous action of  $S$
 on a  compact convex subset $K$ of a locally convex space.
Then there exists a subset $L_0$ of $K$ which is minimal with respect to being nonempty,
compact, convex and satisfying the following conditions.
	\begin{enumerate}
 	   \item[$\mathbf{(\star 1)}$] There exists a collection $\Lambda = \left\lbrace\Lambda_i: i\in I\right\rbrace$ of  closed subsets of $K$ such that $L_0 =\bigcap \Lambda$, and
    \item[$\mathbf{(\star 2)}$]  for each $x\in  L_0$  there is a left ideal $J_i\subseteq S$ such that $J_i.x\subseteq\Lambda_i$ for each $i\in I$.
	\end{enumerate}
	Furthermore, $L_0$ contains a subset $Y$ that is minimal with respect to being nonempty,   closed and $S$-invariant, i.e. $s.Y\subset Y$ for all $s\in S$.
\end{Lemma}
\begin{comment}
\begin{proof}
We sketch the arguments in the proof of \cite[Theorem 3.1]{HolLau71}.
	By Zorn's lemma, such $L_0$ always exists. For each $x\in L_0$,  let $\Phi$ be the collection of all finite intersections
 of sets in $\left\lbrace\Lambda_i: i\in I\right\rbrace$.  Order $\Phi$ by the reverse set inclusion.
 For any $\alpha=\Lambda_1\cap\Lambda_2\cap\cdots\cap\Lambda_n\in\Phi$,
  choose left ideals $J_i$ such that $J_i.x\subseteq\Lambda_i$ for $i=1,\ldots, n$.
  By the right reversibility, there exists $s_\alpha\in\bigcap_{i=1}^n\overline{J_i}$. Thus,
  $Ss_\alpha.x\subseteq\alpha$. Let $z$ be a cluster point of the net
  $\left\lbrace s_\alpha.x\right\rbrace_{\alpha\in\Phi}$.
  Then $\overline{Sz}$ is a closed $S$-invariant subset of $L_0$. By Zorn's lemma, there exists a minimal subset
  $Y\subseteq\overline{Sz}\subseteq L_0$ with respect to being nonempty, closed, and $S$-invariant.
\end{proof}
\end{comment}

\begin{Lemma}[Muoi and Wong {\cite[Lemma 2.3]{MuoiWong2020}}]\label{lemma2}
	The subset $Y$ of $K$ in Lemma \ref{lemma1} is $S$-preserving, i.e., $s.Y=Y$ for all $s\in S$,
if we suppose further  that
$\operatorname{LUC}(S)$ has a $\operatorname{LIM}$ and the action is jointly continuous on $K$.
\end{Lemma}

\begin{proof}
For completeness, we sketch here the argument in the proof of \cite[Lemma 2.3]{MuoiWong2020}.
	For each pair of $y\in Y$ and $f\in C(Y)$, define a function $R_yf\in l^\infty(S)$ by $R_yf(s)=f(s.y)$.  Following
the proof of \cite[Theorem 1]{MIT70}, with the joint continuity of the action we can show that $R_yf\in \operatorname{LUC}(S)$.
\begin{comment}
	Indeed, if $R_yf\notin \operatorname{LUC}(S)$, and there is a net $\left\lbrace t_\lambda\right\rbrace\subseteq S$ such that
$t_\lambda\to t$ but
$\|(l_{t_\lambda}(R_yf)-l_t(R_yf)\|\nrightarrow 0$.
There exist  an $\varepsilon_0>0$, a subnet
$\left\lbrace t_{\lambda_j}\right\rbrace$ of $\left\lbrace t_\lambda\right\rbrace$, and a net $\{s_{\lambda_j}\}$ in $S$ such that
$$
|l_{t_{\lambda_j}}(R_yf)(s_{\lambda_j})-l_t(R_yf)(s_{\lambda_j})|=|f(t_{\lambda_j}s_{\lambda_j}.y)-f(ts_{\lambda_j}.y)|\geq\varepsilon_0.
$$
	Since $Y$ is compact, $\left\lbrace s_{\lambda_j}.y\right\rbrace_j\subseteq Y$ has a  subnet
$\left\lbrace s_{{\lambda_j}_t}.y\right\rbrace_t$ that converges to $z_0\in Y$.
Replacing the nets with their subnets, we can assume that there is a net $\{s_\lambda\}$ in $S$ such that
$$
|f(t_{\lambda}s_\lambda.y)-f(ts_\lambda.y)|\geq\varepsilon_0, \quad\forall \lambda,
$$
and $s_\lambda.y\to z_0$ in $Y$.
By the jointly continuity of the action and the
continuity of $f$ on $Y$, we have
\begin{align}\label{eq:jc-needed}
\varepsilon_0\leq \lim_{\lambda}|f(t_{\lambda}.(s_{\lambda}.y))-f(t.(s_{\lambda}.y))|=|f(t.z_0)-f(t.z_0)|=0.
\end{align}
	The contradiction shows that $R_yf\in \operatorname{LUC}(S)$.
\end{comment}
	
	Let $m$ be a $\operatorname{LIM}$ of  $\operatorname{LUC}(S)$.
Define a linear functional $\psi$ on $C(Y)$ by $\psi(f)=m(R_yf)$.   Observe
	$$
|\psi(f)|\leq \|R_yf\|=\sup_{s\in S}|f(s.y)|\leq \|f\|_{C(Y)}, \quad\forall f\in C(Y).
$$
Since $\psi(1)=1$, we have $\|\psi\|=1$. For each $t\in S$, notice that
$$
\psi(l_tf)=m(R_y(l_tf))=m(l_t(R_yf))=m(R_yf)=\psi(f).
$$
Thus, $\psi$ is a left invariant mean of $C(Y)$. Let $\mu$ be the probability measure on $Y$ corresponding to $\psi$.
Define $Y_0$ to be the support of
$\mu$, i.e.,
$$
Y_0=\operatorname{supp}(\mu)=\bigcap\left\lbrace F\subseteq Y: F \ \text{is closed and}\ \mu(F)=1\right\rbrace.
$$

We are going to see that $s.Y_0=Y_0$ for all $s\in S$.
For every Borel subset $A$ of $Y$ and $s\in S$, define $L_s(x)=s.x$ and $L_s^{-1}A=\left\lbrace x\in Y: s.x\in A\right\rbrace$.
Since $\psi$ is left invariant, we have
$$
\mu(A)= \int_Y \mathbf{1}_A(x)\, d\mu = \int_Y \mathbf{1}_A(s.x) \,d\mu = \int_Y \mathbf{1}_{L_s^{-1}A}(x) \, d\mu = \mu(L_s^{-1}A).
$$
Because $L_s^{-1}Y_0$
is closed and {$\mu(L_s^{-1}Y_0)=\mu(Y_0)=1$}, we have  $Y_0\subseteq L_s^{-1}Y_0$ or $sY_0\subseteq Y_0$.
On the other hand, $\mu(sY_0)=\mu(L_s^{-1}(sY_0))\geq\mu(Y_0)=1$.
This implies $Y_0\subseteq sY_0$, and thus $sY_0=Y_0$.
By the minimality of $Y$, we conclude that $Y=Y_0$ is $S$-preserving.
\end{proof}

\begin{Lemma}\label{lem:3}
	Let $(X,d)$ be a Fr\'{e}chet space with a $d$--admissible locally convex  topology $\tau$.
Let $S$ be a right reversible semitopological semigroup. Assume that an action of $S$ on a $d$--separable and
 $\tau$--compact subset $Y$ of $X$ is separately $\tau$ continuous and
  super asymptotically $d$--nonexpan\-sive. Suppose $Y$ is minimal with
 respect to being $\tau$--closed and $S$-invariant. Suppose further that  there exists a nonempty $\tau$--closed subset $F$ of $Y$ such that
 $F\subset sF$ for all $s\in S$. Then $F$ is $d$--compact. Especially, if $Y$ is $S$--preserving then $Y$ is $d$--compact.
\end{Lemma}
\begin{proof} We follow an idea from the proof of
 \cite[Lemma 3.1]{LauTaka09} in which  nonexpan\-sive actions are considered instead.
	Let $Z$ be the set of all points of continuity of the identity mapping from $(Y,\tau)$ to $(Y,d)$. By \cite[Corollary 1.3]{Nami67},
$Z$ is a dense $G_\delta$ subset of $(Y,\tau)$. Let $U=\{x\in X: d(x,0)<\varepsilon\}$
be a $d$--neighborhood of 0 where $\varepsilon>0$. For each $z\in Z$, there exists a $\tau$--neighborhood $V$ of
$0$ such that $(z+V)\cap Y\subset (z+U)\cap Y$.
	 We can choose a $\tau$--neighborhood $V_1$ of $0$ such that $V_1+V_1\subseteq V$.
Since the $d$--topology is stronger than $\tau$, we have $V_1$ contains a $d$--open neighborhood $U_1$ of $0$. Assume
$U_1=\{x\in X: d(x,0)<\delta\}$ for some $\delta>0$. Since $Y$ is $d$--separable, there exists
a sequence $\{x_i: i\in \mathbb{N}\}$ in $Y$ such that
$$
Y=\bigcup\left\lbrace (x_i+U_1)\cap Y: i\in\mathbb{N}\right\rbrace.
$$

	Since the action is super asymptotically $d$--nonexpan\-sive, for each given $r_0\in S$, there exists a left ideal
$I_1=I_{x_1}^{r_0}$ of $S$ such that $d(sr_0.x_1,sr_0.y)\leq d(x_1,y)$ for all  $s\in I_1$ and $y\in Y$.
Since $I_1r_0.x_1$ is $S$-invariant in $Y$, by the minimality of $Y$, its $\tau$--closure must be exactly $Y$. Thus, there exists a $s_1\in I_1$ such that $s_1r_0.x_1\in (z+V_1)\cap Y$.
Let $r_1=s_1r_0$, we have $r_1.x_1\in (z+V_1)\cap Y$ and  $d(r_1.x_1,r_1.y)\leq d(x_1,y)$ for all  $y\in Y$.
	
	Similarly, there exists a left ideal $I_2=I_{x_2}^{r_1}$ of $S$ such that $d(sr_1.x_2,sr_1.y)\leq d(x_2,y)$ for
all  $s\in I_2$ and all $y\in Y$. There exists $s_2\in I_2$ such that
$s_2r_1.x_2\in (z+V_1)\cap Y$. Let $r_2:=s_2r_1$, we have $r_2.x_2\in (z+V_1)\cap Y$ and $d(r_2.x_2,r_2.y)\leq d(x_2,y)$ for all  $y\in Y$.
		By induction, we can choose a sequence $\left\lbrace r_i: i\in \mathbb{N}\right\rbrace$ in $S$ such that
\begin{gather*}
r_i.x_i\in (z+V_1)\cap Y,\quad  r_i=s_is_{i-1}\cdots s_1r_0,
\intertext{and}
 d(r_i.x_i,r_i.y)\leq d(x_i,y), \quad\forall  y\in Y \mbox{ and } i\geq 1.
\end{gather*}
	
	For each $y\in (x_i+U_1)\cap Y$, we have $r_iy=(r_iy-r_ix_i)+r_ix_i$ where $r_ix_i\in (z+V_1)\cap Y$ and
$d(r_ix_i,r_iy)\leq d(x_i,y)=d(x_i-y,0)<\delta$. Thus,
$$
r_i((x_i+U_1)\cap Y)\subseteq (z+V_1+U_1)\cap Y\subseteq (z+V)\cap Y.
$$
	We rewrite the action $r.x$ in the form of $L_rx$.
Then $(x_i+U_1)\cap Y\subseteq L_{r_i}^{-1}((z+V)\cap Y)$, where $L_{r_i}^{-1}((z+V)\cap Y)$ is $\tau$--open by the separate
$\tau$--continuity of the action.
By the $\tau$--compactness, we can cover $Y$ by  finitely many such open sets.  Let
$$
Y=\bigcup_{i=1}^{n}L_{r_i}^{-1}((z+V)\cap Y).
$$
	
	It follows from the super asymptotic $d$--nonexpan\-siveness of the action that
 there exist closed left ideals $J_i=J_{z}^{t_i}$, where $t_i=s_{n+1}s_{n}\cdots s_{i+1}$ and $i=1,\ldots, n$, such that
 $d(st_i.z,st_i.y)\leq d(z,y)$ for all  $y\in Y$ and $s\in J_i$. Since $S$ is right reversible, there exists $t_0\in\cap_{i=1}^{n}J_i $. Consequently,
\begin{align}\label{eq:dist-bar}
 d(t_0t_i.z,t_0t_i.y)\leq d(z,y) \quad\text{for all}\  y\in Y \mbox{ and } \ i=1,\ldots, n.
\end{align}

By the assumption, $F\subset sF$ for all $s\in S$, we have
\begin{equation*}
\begin{array}{rl}
F\subset L_{t_0}L_{r_{n+1}}F\subset L_{t_0}L_{r_{n+1}}Y &=L_{t_0}L_{r_{n+1}}\left\lbrace \bigcup_{i=1}^{n}L_{r_i}^{-1}((z+V)\cap Y)\right\rbrace\\
&\subseteq \bigcup_{i=1}^{n}\left\lbrace L_{t_0}L_{s_{n+1}\cdots s_{i+1}}((z+U)\cap Y)\right\rbrace\\
&= \bigcup_{i=1}^{n}\left\lbrace L_{t_0t_i}((z+U)\cap Y)\right\rbrace\\
&\subseteq \bigcup_{i=1}^{n}\left\lbrace (L_{t_0t_i}z+U)\cap Y)\right\rbrace\\
&\subseteq \bigcup_{i=1}^{n}\left\lbrace L_{t_0t_i}z+U\right\rbrace.
\end{array}
\end{equation*}
The second last inclusion above follows from \eqref{eq:dist-bar}. This proves that the $d$--closed subset $F$ can be covered by
 finitely many translates of any given
 $d$--neighborhood $U$ of $0$.  In other words, $F$ is totally bounded. Since $(X,d)$ is complete, $F$ is $d$--compact.
	\end{proof}

We provide below a metric version of  DeMarr's Lemma \cite[Lemma 1]{DeMarr63}.

\begin{Lemma}\label{DemarrFrechet}
Let $(X,d)$ be a metrizable locally convex space with a ``good'' metric $d$.
Let $Y$ be a compact subset of $X$ containing more than one point. Then there exists a point $u$ in
the convex hull $\operatorname{conv}(Y)$ of $Y$ such that
 $$\sup\left\lbrace d(u,y): y\in Y\right\rbrace<\sup\left\lbrace d(x,y): x,y\in Y\right\rbrace.$$
\end{Lemma}

\begin{proof}
Let
$r=\mbox{diam} (Y)=\sup\left\lbrace d(x,y): x, y\in Y\right\rbrace$. Then $x-y\in\bar{B}_r(0)$ for all $x, y\in Y$.
For each $\lambda\in\mathbb{Q}\cap (0,+\infty)$, the metric balls $B_\lambda(0)$ is open, convex and balance. Thus
\begin{align*}
B_\lambda(0)&=\left\lbrace x\in X: d(x,0)<\lambda\right\rbrace\\
&=\left\lbrace x\in X: q_\lambda(x)<1\right\rbrace
=\left\lbrace x\in X: q_\lambda(\lambda x)<\lambda\right\rbrace,
\end{align*}
where the seminorm $q_\lambda$ is defined in \eqref{MinkowskiSeminorm}.
This implies
$$
d(x,0)=q_\lambda(\lambda x)\ \text{for each}\ \lambda\in\mathbb{Q}\cap (0,+\infty)\ \text{and}\ x\in B_\lambda(0).
$$

Choose a $\lambda\in\mathbb{Q}\cap (0,+\infty)$ such that $\lambda>r$, we have $d(x,y)=p_\lambda(\lambda(x-y))$ for all $x, y\in Y$.
Since $Y$ is compact, there exists a finite subset $M=\{x_1, x_2,\ldots,x_n\}\subset Y$ which is maximal with respect to being that $d(x_i,x_j)=r$ for all $i\neq j$.

Let $u=\frac{1}{n}\sum_{i=1}^{n}x_i\in \operatorname{conv}(Y)\subseteq\bar{B}_r(0)$ and $y_0\in Y$ such that $r_0=d(u,y_0)=\max_{y\in Y}d(u,y)$.
Suppose on the contrary that $r_0=r$.  Then
$$
r=d(u,y_0)=p_\lambda(\lambda(u-y_0))\leq\frac{1}{n}\sum_{i=1}^{n}p_\lambda(\lambda(x_i-y_0))=\frac{1}{n}\sum_{i=1}^{n}d(x_i,y_0)\leq r.
$$
This drives $d(x_i,y_0)=r$ for all $i=1,\ldots,n$. By the maximality of $M$, we have $y_0=x_{i_0}$ for some $i_0\in \{1,\ldots,n\}$. This
conflicts with the fact that $d(x_{i_0},y_0)=r>0$.
\end{proof}

Following the idea in the proofs of \cite[Theorem 3.1]{HolLau71}  and \cite[Theorem 4.2]{AAR2018}, we are able to prove our main result.

\begin{proof}[Proof of Theorem \ref{mainThm}]
By Lemmas \ref{lemma2} and \ref{lem:3}, the nonempty
 $S$-preserving set $Y$ given in Lemma \ref{lemma1} is a $d$--compact subset of
the Fr\'{e}chet space $X$. Consequently, the $d$--topology agrees with  $\tau$  on $Y$.
If  $Y$ contains exactly one point  then we are done. Otherwise, let
$$
r=\mbox{diam} (Y)=\sup\left\lbrace d(x,y): x, y\in Y\right\rbrace.
$$
By Lemma \ref{DemarrFrechet}, there exists a  $u\in\operatorname{conv}(Y)$ such that
 $$r_0=\sup\left\lbrace d(u,y): y\in Y\right\rbrace<\sup\left\lbrace d(x,y): x,y\in Y\right\rbrace=r.$$

Let $0<\varepsilon<r-r_0$.
For each $\Lambda\in \left\lbrace\Lambda_i: i\in I\right\rbrace$  in Lemma \ref{lemma1}, we
set
\begin{align*}
 N_{\varepsilon,\Lambda}&= \bigcap_{y\in Y} \{x\in \Lambda : d(x,y)\leq  r_0+\varepsilon\}
\intertext{and}
N_0&=\bigcap\left\lbrace N_{\varepsilon, \Lambda_i}:  \ i\in I\right\rbrace= L_0\cap \bigcap_{y\in Y}\bar{B}[y, r_0+\varepsilon],
\end{align*}
where $\bar{B}[y, \delta]$ denotes  the closed ball centered at $y$ of radius $\delta$.

We show that $N_0$  satisfies conditions $\mathbf{(\star1)}$ and $\mathbf{(\star2)}$. Indeed, every $N_{\varepsilon,\Lambda_i}$ is $d$--compact.
Thus $N_0$ is a $d$--compact subset of $L_0$, and contains $u$.
For each $x\in N_0$ and  $i\in I$, there exists a left ideal $I\subseteq S$ such that $I.x\subseteq \Lambda_i$.
By the super asymptotic $d$--nonexpan\-siveness of the action, for each $t\in S$ there exists a left ideal $I_{x}^t$
such that  $d(st.y,st.x)\leq d(y,x)$ for all  $y\in K$ and $s\in I_{x}^t$.
By the right reversibility of $S$, there exists a  $t_0\in\overline{I}\cap\overline{I_{x}^t t}$. Since $\Lambda_i$ is $\tau$--closed, $St_0.x\subseteq \Lambda_i$.
   Consider a net $s_\lambda\in I_{x}^t$ such that $s_\lambda t\to t_0$.
From  $d(ss_\lambda t.y,ss_\lambda t.x)\leq d(y,x)\leq r_0+\varepsilon$ for all $\lambda$, $y\in Y$ and $s\in S$,
we have $d(st_0.y,st_0.x)\leq d(y,x)\leq r_0+\varepsilon$.
Since $Y\subset st_0.Y$, we have $d(y', st_0.x)\leq r_0+\varepsilon$ for all $y'\in Y$.
In other words, there exists a left ideal $J=St_0$ of $S$
such that $J.x\subseteq N_{\varepsilon, \Lambda_i}$.  Consequently, the nonempty, $\tau$--compact,
convex subset $N_0$   satisfies conditions $\mathbf{(\star1)}$ and $\mathbf{(\star2)}$.

By the minimality of $L_0$, we have
$Y\subseteq L_0= N_0 \subseteq \bigcap_{y\in Y}\bar{B}[y, r_0+\varepsilon]$.
This gives us a contradiction that $\operatorname{diam}(Y) \leq r_0+\epsilon < r$.
Therefore, $Y$ contains a unique point and it is the common fixed point for the action of $S$ on $K$.
\end{proof}

\begin{comment}
If the assumption of super asymptotic $d$--nonexpan\-siveness in Theorem \ref{mainThm} is replaced by the super asymptotic separately
 $Q$--nonexpan\-siveness, where $Q$ is any countable family of seminorms defining the metric topology
such that each $q\in Q$ is lower semicontinuous with respect to the topology $\tau$.  Then, with similar arguments, we have the following result.
\end{comment}

When the semigroup $S$ is right reversible,
the following proposition of us in \cite{MuoiWong2020} is an extension of a result by Lau and Zhang
\cite[Theorem 6.2]{LauZhang12}. Their result holds for norm nonexpansive and jointly weak*
continuous actions on a weak* compact convex subset of a Banach dual space.
Here, in this paper, we have a new proof.  In fact, it is a direct consequence of Theorem \ref{mainThm}, by noting that
for a Banach space (resp.\  a Banach dual space) the weak (resp.\ weak*) topology is $\|\cdot\|$--admissible.

\begin{Proposition}[Muoi and Wong {\cite[Theorem 2.5]{MuoiWong2020}}]\label{MWFjsep}
	Let $S$ be a right reversible semitopological semigroup.
	Assume that $\operatorname{LUC}(S)$ has a $\operatorname{LIM}$.
Then $S$ has the following fixed point property.
	
	\begin{quote}
		$\mathbf{(F^{sup}_{jw^*c,Nsep})}$ Every  super asymptotically non\-expan\-sive and jointly weak* continuous  action   of  $S$
		on a weak* compact convex and norm separable subset of  a Banach dual space has a common fixed point.
	\end{quote}
\end{Proposition}

 It is known that every finite Radon measure on a weakly compact subset of a Banach space has a norm separable support
 (see, e.g., \cite[Theorem 4.3, page 256]{Joram72}). This implies that the subset $Y$ in Lemma \ref{lemma2} is norm separable,
 when $(X,\|\cdot\|)$ is a Banach space and $\tau$ is the weak topology of $X$.
 Consequently, Theorem \ref{mainThm} holds also for a nonseparable subset $K$ in this case.
 This provides a new proof of the following result.

\begin{Proposition}[Muoi and Wong {\cite[Theorem 2.1]{MuoiWong2020}}]
	Let $S$ be a right reversible semitopological semigroup.
	Assume that $\operatorname{LUC}(S)$ has a left invariant mean. Then $S$ has the following fixed point property.
	\begin{quote}
		$\mathbf{(F^{\sup}_{jwc})}$ Every  super asymptotically nonexpan\-sive and jointly weakly continuous action of  $S$
		on a weakly compact  convex subset  of a Banach space has a common fixed point.
	\end{quote}
\end{Proposition}

Recall that for a discrete semitopological semigroup $S$,
the condition that $\operatorname{LUC}(S)$ has a $\operatorname{LIM}$ is strictly stronger than
that $S$ is left reversible (see \cite[page 2549]{LauZhang08}), while in general it might not be the case.
%We continue with a property of the action of a left reversible semitopological semigroup that will be used in  Theorem \ref{ThmReversible}.

\begin{Lemma}[Lau and Zhang {\cite[Lemma 3.4]{LauZhang12}}]\label{left-reversibility}
Let S be left reversible semitopological semigroup. Consider an action $S\times Y\mapsto Y$ of $S$ on a compact subset $Y$ of a locally convex space $X$. Then
\begin{itemize}
	\item [(i)] there is a closed subset $F$ of $Y$ such that $F\subset sF$ for all $s\in S$ if the action is separately continuous;
	\item [(ii)] there is a closed subset $F$ of $Y$ such that $sF=F$ for all $s\in S$ if the action is jointly continuous.
\end{itemize}
\begin{proof}
	We give a different proof than \cite[Lemma 3.4]{LauZhang12}, since our approach is more in line with the reasoning in this paper.  Following an idea in \cite{MIT70Rever}, see also \cite[Lemma 4]{HolLau71}, let $F=\bigcap \{sY: s\in S\}$ where each $sY$ is compact.
	Consider any finite collection $\{s_1Y, s_2Y,\ldots,s_nY\}$. By the left reversibility of $S$, there is a $t\in S$ such that $t\in\bigcap_{i=1}^n\overline{s_iS}$. We have
	\begin{equation}\label{eq27}
	\bigcap_{i=1}^ns_iY\supset\bigcap_{i=1}^n\overline{s_i(SY)}\supset\bigcap_{i=1}^n\overline{s_iS}Y\supset tY\neq\emptyset.
	\end{equation}
	It follows that $F$ is nonempty.
	
	We claim that $F\subset sF$ for all $s\in S$. We need to prove that $y\in sF$ whenever $y\in F$ and $s\in S$.
	Consider any finite collection $\{s_1Y, s_2Y,\ldots,s_nY\}$, from \eqref{eq27} we have
	$$\left(L_s^{-1}\{y\}\right)\cap\bigcap_{i=1}^ns_iY\supset\left(L_s^{-1}\{y\}\right)\bigcap tY\neq\emptyset,$$
	since $y\in F\subset stY$ then there is a $x\in Y$ such that $y=stx$, hence $tx\in L_s^{-1}\{y\}$. By the finite intersection property, $L_s^{-1}\{y\}\bigcap F\neq\emptyset$. Consequently, $y\in sF$. That proves our claim.
	
	When the action is jointly continuous, we show that $F$ is $S$-invariant.
	We need to prove that $a.x\in F$ whenever $x\in F$.
	To see this, for any $b\in S$, let $e\in S$ such that $e\in \overline{aS}\cap\overline{bS}$.
	There exist nets $\{c_\lambda\}_\lambda$ and $\{d_\lambda\}_\lambda$ in $S$
	such that $\{ac_\lambda\}_\lambda$, $\{bd_\lambda\}_\lambda$ converge to $e$.  Since
	$x\in F=\bigcap \{sY: s\in S\}$,  there is a $x_\lambda\in Y$ such that $x=c_\lambda x_\lambda$ for every $\lambda$.
	By the weak compactness of $Y$, we can assume $\{x_\lambda\}$ converges to some $x_0\in Y$.
	Therefore,
	\begin{align}\label{ax-joint-cont}
	a.x=(ac_\lambda)x_\lambda\rightarrow e.x_0
	\end{align}
	by the
	joint  continuity of the action. This implies $a.x=e.x_0$.
	Since $bd_\lambda x_0\rightarrow e.x_0$, we have $a.x=e.x_0\in \overline{bS}Y\subset \overline{bSY}\subset \overline{bY} =  bY$,
	since  $bY$ is compact. Consequently, $sF=F$ for all $s\in S$.	
\end{proof}	
\end{Lemma}
The following result supplements Theorem \ref{mainThm}.  The key point in its proof is that we can bypass Lemma \ref{lemma2}.

\begin{Theorem}\label{ThmReversible}
		Let $S$ be a reversible semitopological semigroup and $(X,d)$ be a  Fr\'{e}chet space with a ``good'' metric $d$.
Let $\tau$ is a $d$--admissible locally convex topology of $X$. Then $S$ has the following fixed point property.
	\begin{quote}
		$\mathbf{(F^{\sup}_{sc,sep})}$ Every  super asymptotically $d$--nonexpan\-sive and separately $\tau$ continuous  action of  $S$
		on a $d$--separable and $\tau$--compact convex subset $K$ of $X$  has a common fixed point.
	\end{quote}
\end{Theorem}
\begin{proof}
	By  Lemma \ref{lemma1}, there is a subset $L_0$ of $K$ which is minimal with respect to being nonempty, $\tau$--compact,
	convex and satisfying conditions $\mathbf{(\star1)}$ and $\mathbf{(\star2)}$. Moreover, $L_0$ contains
	a subset $Y$ that is minimal with respect to being nonempty, $\tau$--compact and $S$-invariant. By Lemmas \ref{left-reversibility} and \ref{lem:3}, there is a $d$--compact subset $F$ of $Y$ such that $F\subset sF$ for all $s\in S$.
	The remaining part
follows similarly as in the proof of Theorem \ref{mainThm} where the set $Y$ is replaced by its $d$--compact subset $F$.
\end{proof}

\begin{Theorem}\label{APWAP}
	Let $S$ be a right reversible semitopological semigroup and $(X,d)$ be a  Fr\'{e}chet space with a ``good'' metric $d$.
Let $\tau$ be a $d$--admissible locally convex topology on $X$.
	\begin{itemize}
		\item [(i)] Assume $\operatorname{AP}(S)$ has a $\operatorname{LIM}$. Then every super asymptotically $d$--nonexpan\-sive,
separately $\tau$ continuous, and $\tau$ equicontinuous  action of  $S$
		on a $d$--separable and $\tau$--compact convex subset $K$ of $X$  has a common fixed point.
		\item [(ii)] Assume $\operatorname{WAP}(S)$ has a $\operatorname{LIM}$. Then every
super asymptotically $d$--nonexpan\-sive,  separately $\tau$ continuous, and $\tau$ quasi-equicontinuous action of  $S$
		on a $d$--separable and $\tau$--compact convex subset $K$ of $X$  has a common fixed point.
	\end{itemize}
	
\end{Theorem}

\begin{proof}
	These are direct consequences of
\cite[Lemma 3.1]{Lau73Rocky}, \cite[Theorem 3.4]{LauZhang08}, and the proof of Theorem \ref{mainThm}.
	\end{proof}
	
\begin{Corollary}\label{corr.Normal}
  Let $S$ be a semitopological semigroup as well as a normal space. Let $(X,d)$ be a  Fr\'{e}chet space with a ``good'' metric $d$.
  Let $\tau$ be a $d$--admissible locally convex topology on $X$.
	 Assume that $\operatorname{CB}(S)$ has an invariant mean. Then $S$ has the fixed point property $\mathbf{F^{\sup}_{sc,sep}}$.
	
\end{Corollary}
\begin{proof}
It is  known that if $S$ is normal and $CB(S)$ has a right invariant mean then $S$ is right reversible.
\begin{comment}
We sketch the proof here for completeness. Suppose $I_1$ and $I_2$ are disjoint
closed left ideals of $S$. Since $S$ is normal, there exists an $f\in \operatorname{CB}(S)$ such that $f= 1$ on $I_1$ and $f= 0$
 on $I_2$. For each $s_1\in I_1, s_2\in I_2$, the  {right translations} of $f$ at $s_1, s_2$ are given by $r_{s_1}f(s)=f(ss_1)=1$ and $r_{s_2}f(s)=f(ss_2)=0$ for all $s\in S$. Then $m(r_{s_1}f)=m(1)=1$ and $m(r_{s_2}f)=m(0)=0$, conflicting with the right translation invariant property
 of $m$.
\end{comment}
The assertion  follows similarly as in proving Theorem \ref{ThmReversible}.
\end{proof}

We now discuss  $Q$--nonexpansive actions.
Let $(X,Q)$ be a  Fr\'{e}chet space in which $Q=\{q_n: n\in\mathbb{N}\}$
is a countable family of seminorms defining the metric topology.
A locally convex topology $\tau$ on $X$ is said to be \emph{$Q$--admissible} (\cite{LauTaka09}) if
$\tau$ is weaker than the $Q$--topology while every seminorm $q_n$ in $Q$ is $\tau$--lowersemicontinuous.

Note  that
one cannot use \eqref{sumMetric} to define a metric and apply Theorem \ref{mainThm} to get the following result,
  as the metric so defined might not be ``good''.

\begin{Corollary}\label{cor:unif-Q}
Let $S$ be a right reversible semitopological semigroup. Let $(X,Q)$ be a  Fr\'{e}chet space
with a $Q$--admissible locally convex topology $\tau$.
Assume that $\operatorname{LUC}(S)$ has a left invariant mean. Then $S$ has the following fixed point property.
\begin{quote}
	 Every  super asymptotically uniformly $Q$--nonexpan\-sive and jointly $\tau$ continuous  action of  $S$
	on a $\tau$--compact convex and $Q$--separable subset $K$ of $X$  has a common fixed point.
\end{quote}
\end{Corollary}
\begin{proof}
Let $Q=\{q_n: n=1,2,\ldots\}$.  Without loss of generality, by summing and scaling,
we can assume that $4q_n(x)\leq q_{n+1}(x)$ for all $x\in X$ and $n=1,2,\ldots$.
Then the absolutely convex open sets $V_n=\{x\in X: q_n(x) < 1\}$, $n=1,2,\ldots$, constitute a local basis of zero of $X$,
 such that $V_{n+1} + V_{n+1} + V_{n+1} + V_{n+1} \subseteq V_n$ for $n=1,2,\ldots$.
 Let $D$ be the set of rational numbers $r$ in $(0,1)$ such that $r=\sum_{n} c_n(r)2^{-n}$,
  where the binary digit $c_n(r)$ assumes either $0$ or $1$, and among them only finitely many $c_n(r)$ are $1$.
Let $A(r)=c_1(r)V_1 + c_2(r)V_2 + \cdots$ for any rational number $r\in D$ and $A(r)=X$ for $r\geq 1$.
Following the proof of  \cite[Theorem 1.24]{Rudin91}, we can define a ``good'' metric $d$ of $X$ defining its topology such that
the open metric balls $B_\delta(0)=\{x\in X: d(0,x) < \delta\} = \bigcup \{ A(r): r\in D, 0<r< \delta\}$. It is not difficult
to see that   $Q$--admissible locally convex topologies of $X$ are exactly those being $d$-admissible,
 and super asymptotically (resp.\ asymptotically) uniformly $Q$--nonexpansive actions on any subset $K$ of $X$ are exactly those being
super asymptotically (resp.\ asymptotically)  $d$--nonexpansive.  Consequently, we can apply Theorem \ref{mainThm}.
\end{proof}

However, we have an even better version of Corollary \ref{cor:unif-Q} in the following, which works for super asymptotically
\emph{separately} $Q$--nonexpansive actions.

\begin{Theorem}\label{thm:main-Q}
 Let $S$ be a right reversible semitopological semigroup. Let $(X,Q)$ be a  Fr\'{e}chet space
with a $Q$--admissible locally convex topology $\tau$.
Assume that $\operatorname{LUC}(S)$ has a left invariant mean. Then $S$ has the following fixed point property.
\begin{quote}
	$\mathbf{(F^{\sup}_{jc,sQsep})}$ Every  super asymptotically separately $Q$--nonexpan\-sive and jointly $\tau$ continuous  action of  $S$
	on a $\tau$--compact convex and $Q$--separable subset $K$ of $X$  has a common fixed point.
\end{quote}
\end{Theorem}
\begin{proof}
 	We adapt  the proof of Theorem \ref{mainThm} to the locally convex space setting.
	By Lemma \ref{lemma1}, there is a subset $L_0$ of $K$ which is minimal with respect to being nonempty, compact,
	convex and satisfying conditions $\mathbf{(\star1)}$ and $\mathbf{(\star2)}$. Moreover, $L_0$ contains
	a subset $Y$ that is minimal with respect to being nonempty, compact and $S$-invariant. Following Lemma \ref{lemma2}, we see that $Y$ is $S-$preserving.  With similar arguments as in Lemma \ref{lem:3}, we have $Y$ is $Q$-compact.
	
	If  $Y$ contains exactly one point  then we are done. Otherwise, there exists a seminorm $q$ in $Q$ such that
	$$
	r=\mbox{diam}_q (Y):=\sup\left\lbrace q(x-y): x, y\in Y\right\rbrace>0.
	$$	
	Following the arguments in the proof of Lemma \ref{DemarrFrechet}, noting also \cite[Lemma 1]{DeMarr63}, we see that there exists a  $u\in\operatorname{conv}(Y)$ such that
	$$r_0=\sup\left\lbrace q(u-y): y\in Y\right\rbrace<\sup\left\lbrace q(x-y): x,y\in Y\right\rbrace=r.$$
Let $0<\varepsilon<r-r_0$.
For each $\Lambda\in \left\lbrace\Lambda_i: i\in I\right\rbrace$  in Lemma \ref{lemma1}, we
set
\begin{align*}
N_{\varepsilon,\Lambda}&= \bigcap_{y\in Y} \{x\in \Lambda : q(x-y)\leq  r_0+\varepsilon\}
\intertext{and}
N_0&=\bigcap\left\lbrace N_{\varepsilon, \Lambda_i}:  \ i\in I\right\rbrace= L_0\cap \bigcap_{y\in Y}\bar{B}_q[y, r_0+\varepsilon],
\end{align*}
where $\bar{B}_q[y, \delta]$ denotes  the $q$--closed ball centered at $y$ of radius $\delta$. Following the arguments in the last part of the proof of Theorem \ref{mainThm}, we see that $N_0$ is nonempty, compact,
convex and satisfies conditions $\mathbf{(\star1)}$ and $\mathbf{(\star2)}$.

By the minimality of $L_0$, we have
$Y\subseteq L_0= N_0 \subseteq \bigcap_{y\in Y}\bar{B}_q[y, r_0+\varepsilon]$.
This gives us a contradiction that $\operatorname{diam}_q(Y) \leq r_0+\epsilon < r$.
Therefore, $Y$ contains a unique point and it is the common fixed point for the action of $S$ on $K$.
\end{proof}

In a similar manner, we will get the $Q$--nonexpansive versions of Theorems \ref{ThmReversible} and \ref{APWAP}, and Corollary \ref{corr.Normal}.

\begin{Theorem}\label{thm:AP-Q}
  	Let $S$ be a right reversible semitopological semigroup.
   Let $(X,Q)$ be a  Fr\'{e}chet space
with a $Q$--admissible locally convex topology $\tau$.
	\begin{itemize}
\item[(i)] Assume $S$ is reversible.  Then every  super asymptotically separately
$Q$--nonexpan\-sive and separately $\tau$ continuous  action of  $S$
		on a $Q$--separable and $\tau$--compact convex subset $K$ of $X$  has a common fixed point.
		\item [(ii)] Assume $\operatorname{AP}(S)$ has a $\operatorname{LIM}$. Then every separately $\tau$--continuous, $\tau$--equicontinuous and super asymptotically separately $Q$--nonexpan\-sive action of  $S$ on a $\tau$--compact convex subset $K$ of $X$ has a common fixed point.
		\item [(iii)] Assume $\operatorname{WAP}(S)$ has a $\operatorname{LIM}$. Then every $\tau$--separately continuous,
$\tau$--quasi-equicon\-tin\-uous and super asymptotically separately $Q$--nonexpan\-sive action of $S$
		on a $\tau-$compact convex subset $K$ of $X$  has a common fixed point.
	\end{itemize}	
\end{Theorem}

\section{Fixed point properties on locally convex spaces}\label{Sect3}

In this section, we consider  super asymptotically separately $Q$--nonexpan\-sive actions
of a semitopological semigroup $S$ on a
compact convex set $K$ in a general locally convex space  $(X,Q)$,
 in which $Q$ is a family
of seminorms defining the topology.
Note that we do not assume the metrizability or the completeness of $X$, and we do not assume the separability of $K$  either.

 We note that the results in \cite{HolLau71,LauZhang08,LauTaka09},
though  stated for asymptotically uniformly $Q$--nonexpan\-sive actions,
 hold indeed with the same proofs for
asymptotically \emph{separately} $Q$--nonexpan\-sive actions.
However, the following example tells us that the various separately $Q$--nonexpansiveness are strictly weaker than their
uniform versions.

\begin{Example}	\label{eg:sep-not-unif}
	Let $X$ be the Fr\'{e}chet space of all scalar
sequences $x=(x_n)$
%$X=\{(x_n): x_n\in\mathbb{C}\}$
equipped with the topology of coordinate-wise convergence; namely, it is the topology
defined by the countable family $Q$ of seminorms $q_n(x)=|x_n|$ for $n\in\mathbb{N}$.
Let $K$ be the compact convex subset of $X$ defined by
$$
K=\{(x_n)\in X: |x_n|\leq 1 \mbox{ for all } n\in\mathbb{N}\}.
$$
Let $\mathbb{N}_0$ be the additive discrete semigroup of nonnegative integers acting on $K$ by right shifts:
$$
(k.x)_j=0 \mbox{ for all } j\leq k, \quad \mbox{ and } (k.x)_j=x_{j-k} \mbox{ elsewhere}.
$$
Note that  any (necessarily two-sided) ideal $J$ of $\mathbb{N}_0$ assumes the form $J=s+\mathbb{N}_0$ for   $s=\min J$.	
	
		This action is not asymptotically uniformly  $Q$--nonexpan\-sive.
Indeed, consider any $x=(x_n), y=(y_n)$ in $K$ such that $|x_n-y_n|$ is strictly decreasing.
For any left ideal $J=s+\mathbb{N}_0$ of $\mathbb{N}_0$,  choosing $n>k>s$  we have
$$
q_n(k.x - k.y)  = | x_{n-k} - y_{n-k} | > | x_n - y_n | = q_n(x-y).
$$
	Thus the action is not asymptotically uniformly $Q$--nonexpan\-sive.
	
	On the other hand, this action is super asymptotically separately $Q$--nonexpan\-sive. Indeed, for any
  seminorm $q_n$, we choose the ideal $J=n+\mathbb{N}_0$.  Then for any $x$ in $K$ and any $t\in\mathbb{N}_0$, we have
% For any $s\in J$, we have $s+t>n$, and thus
$$
q_n((s+t).x - (s+t).y) = | ((s+t).x)_n - ((s+t).y)_n| = 0 \leq q_n(x-y),\quad\forall s\in J, \, \forall y\in K.
$$
%	for all $y\in K$.
\end{Example}

\begin{Theorem}\label{LCSsetting}
	Let $S$ be a right reversible semitopological semigroup and  $(X,Q)$ be a   locally convex space.
	\begin{enumerate}[(i)]
%\item Assume $S$ is reversible.  Then every  super asymptotically separately
%$Q$--nonexpan\-sive and separately   continuous  action of  $S$
%		on a $Q$--separable and  compact convex subset $K$ of $X$  has a common fixed point.

		\item   Assume $\operatorname{LUC}(S)$ has a $\operatorname{LIM}$. Then
		every jointly continuous super asymptotically separately $Q$--nonexpan\-sive action of $S$ on a compact convex subset $K$ of $X$ has a common fixed point.
		\item   Assume $\operatorname{AP}(S)$ has a $\operatorname{LIM}$. Then every separately continuous, equicontinuous and super asymptotically separately $Q$--nonexpan\-sive action of  $S$ on a compact convex subset $K$ of $X$ has a common fixed point.
		\item   Assume $\operatorname{WAP}(S)$ has a $\operatorname{LIM}$. Then every separately continuous, quasi-equicon\-tin\-uous and super asymptotically separately $Q$--nonexpan\-sive action of $S$
		on a compact convex subset $K$ of $X$  has a common fixed point.
	\end{enumerate}	
\end{Theorem}
\begin{proof}
The assertions follow from arguments similar to those in the proofs of Theorems \ref{thm:main-Q} and \ref{thm:AP-Q}.
Note that we do not need Lemma \ref{lem:3}, while its conclusion holds automatically as the $\tau$--topology coincides with the $Q$--topology
 in the current setting.
\end{proof}

\begin{comment}
\begin{Remark}
	In fact, when only one topology of locally convex space $X$ is considered, we have the same results as in Theorem \ref{LCSsetting} with a weaker assumption of nonexpan\-siveness that is \textit{asymptotically nonexpan\-sive type}  (see \cite{SGF18}). However, in the setting of section 2, it might not the case.
\end{Remark}
\end{comment}

The following result supplements Theorem \ref{ThmReversible} in the general locally convex space setting.

\begin{Theorem}\label{thm:LCS-Reversible}
	Let $S$ be a reversible semitopological semigroup and $(X,Q)$ be a  locally convex space. Then $S$ has the following fixed point property.
	\begin{quote}
		$\mathbf{(F^{\sup}_{sc})}$ Every super asymptotically separately $Q$--nonexpan\-sive and separately continuous action of $S$ on a compact convex subset $K$ of $X$ has a common fixed point.
	\end{quote}
\end{Theorem}

\begin{Remark}
We do not have the ``two topology $Q-\tau$ version'' of the above results ready.  The difficulty is that we need to assume
$(X,Q)$ to be metrizable to utilize \cite[Corollary 1.3]{Nami67}
in proving Lemma \ref{lem:3}.  Without Lemma \ref{lem:3} we do not know whether the $\tau$-compact set $Y$ is also $Q$--compact.
Thus we are not able to
apply Lemma \ref{DemarrFrechet}, or \cite[Lemma 1]{DeMarr63},
to conclude the existence of a common fixed point of the action.
\end{Remark}

Finally, we show that the notions of
asymptotic nonexpan\-siveness, super asymptotic nonexpan\-siveness and nonexpan\-siveness are strictly different.

\begin{Example}[{Based on \cite[Example]{HolLau71}}]\label{eg:1}
	Let $K=\left\lbrace (r,\theta): 0\leq r\leq 1, 0\leq \theta<2\pi\right\rbrace$ be the closed unit disk in $\mathbb{R}^2$
in polar coordinates and the usual Euclidean norm. Define continuous mappings $f, g$ from $K$ into $K$ such that
	\begin{equation*}
	f(r,\theta)=(r/2,\theta)\quad\text{and}\quad g(r,\theta)=(r,2\theta\ ({\operatorname{mod} 2\pi})).
	\end{equation*}
	Let $S$ be the discrete semigroup generated by $f$ and $g$ under composition.
Then $$
S=\left\lbrace f^ng^m: (m,n)\in \mathbb{N}_0\times\mathbb{N}_0\setminus\left\lbrace (0,0)\right\rbrace\right\rbrace,
$$
where $\mathbb{N}_0 =\{0,1,2,\ldots\}$.
Consider the action of $S$ on $K$  given by
$$
f^ng^m(r,\theta)=(\frac{r}{2^n},2^m\theta\; (\operatorname{mod} 2\pi)).
$$
This action is asymptotically nonexpan\-sive but not super asymptotically nonexpan\-sive (see \cite[Example 2.13]{MuoiWong2020}).
\end{Example}

\begin{Example}[Based on {\cite[Example 3.3(ii)]{AMN2016}}]\label{eg:2}
	 Let $K$ be the closed unit disk in $\mathbb{R}^2$.  Let $f$ be any continuous but not nonexpan\-sive function
from $[-1,1]$ into $[-1,1]$ such that $f(0)=0$. Consider the  map $T:K\to K$ defined by $T(x_1,x_2)=(f(x_2),0)$.
Then $T$ is not nonexpan\-sive and $T^n=0$ for all $n\geq 2$.
Define an action of the reversible discrete additive semigroup $\mathbb{N}$ on $M$ by $n.(x_1,x_2) = T^n(x_1,x_2)$.
It is plain that this action is super asymptotically nonexpan\-sive but not nonexpan\-sive.
\end{Example}

We see, however, in the following that
asymptotically nonexpan\-sive actions of a right reversible
compact  semitopological semigroup $S$ are automatically super asymptotically nonexpan\-sive.

\begin{Proposition}[Based on {\cite[Proposition 2.12(iii)]{MuoiWong2020}}]\label{remark210}
	Let $S$ be a  right reversible  compact semitopological semigroup.  Let $K$ be a subset of a locally convex space $(X,Q)$.
Then every separately continuous and asymptotically separately  (resp.\ uniformly) $Q$--nonexpan\-sive action of $S$ on    $K$
 is  super asymptotically separately (resp.\ uniformly) $Q$--nonexpan\-sive.
\end{Proposition}

\begin{proof}
	Fix an $x\in K$, $t\in S$ and a seminorm $q\in Q$.
For each $y\in K$ there exists a left ideal $I_y$ of $S$ such that $q(s.x-s.y)\leq q(x-y)$ for all $s\in I_y$.
Since the action is separately continuous, we can assume $I_y$ is closed.
For each finite family of closed left ideals $\left\lbrace I_{y_1},\ldots,I_{y_n}\right\rbrace$ of $S$,
it follows from the right reversibility of $S$ that $\bigcap_{i=1}^n I_{y_i}\neq\emptyset$.
It follows from  the compactness of $S$ that
 $I=\bigcap_{y\in K}I_y$ is a nonempty closed left ideal of $S$.
 Clearly, $q(sx-sy)\leq q(x-y)$ for all $s\in I$ and $y\in K$. Since $S$
 is compact, $St$ is a closed left ideal of $S$, and
 thus $St\cap I\neq\emptyset$.
 Then $I^t=\left\lbrace s\in S: st\in I\right\rbrace$ is a nonempty left ideal of $S$. We have $q(st.x-st.y)\leq q(x-y)$ for all $s\in I^t$ and $y\in K$. In other words, the action is super asymptotically separately $Q$-nonexpan\-sive.
Finally, the uniform version follows similarly.
\end{proof}

We end this paper with an open problem for a converse of Theorem \ref{mainThm}, \ref{APWAP} and \ref{thm:main-Q}.

\begin{question}\label{question}
	Let  $S$ be a right reversible semitopological semigroup. Does the fixed point property stated in Theorems \ref{mainThm}, \ref{APWAP} or \ref{thm:main-Q} imply that $\operatorname{LUC(S)}$, $\operatorname{AP(S)}$ or $\operatorname{WAP(S)}$ has a $\operatorname{LIM}$, respectively?.
\end{question}

It is shown in \cite[Proposition 6.5]{LauZhang12} that $\operatorname{AP(S)}$ has a $\operatorname{LIM}$ if $S$ has the following fixed point property.
\begin{quote}
	$\mathbf{(F^{ne}_{jw*,sep})}$ Every norm nonexpansive and jointly weak* continuous action of $S$ on a nonempty norm separable and weak* compact convex subset of a Banach dual space has a common fixed point.
\end{quote}
Note that the fixed point property $\mathbf{(F^{\sup}_{jc,dsep})}$ in Theorem \ref{mainThm} is stronger than  $\mathbf{(F^{ne}_{jw*,sep})}$. Therefore, $\mathbf{(F^{\sup}_{jc,dsep})}$ implies the existence of a $\operatorname{LIM}$ of $\operatorname{AP(S)}$. In particular, the converse of Theorem \ref{mainThm} holds when  $S$ is compact, since $\operatorname{AP(S)}$ and $\operatorname{LUC(S)}$ coincide in this case.

\section*{Acknowledgement}

This research is supported by the Taiwan MOST grant 108-2115-M-110-004-MY2.

The authors would like to thank the referee for encouraging comments and useful suggestions.

\end{document}